\documentclass[11pt,oneside,reqno]{amsart}
\usepackage{amsmath,amsfonts,amssymb,amsthm,mathrsfs}
\usepackage{cite}
\title[Logarithmic Velocity Alignment]{Logarithmic Velocity Alignment on Riemannian Manifolds}

\author[H. Ahn]{Hyunjin Ahn}
\address[Hyunjin Ahn]{\newline Department of Data Technology \newline Myongji University, Seoul 03674, Republic of Korea}
\email{ahj92@mju.ac.kr}

\author[W. Shim]{Woojoo Shim$^*$}
\address[Woojoo Shim]{\newline Department of Mathematics Education \newline Kyungpook National University, Daegu 41566, Republic of Korea}
\email{wjshim@knu.ac.kr}



\thanks{Acknowledgment: The work of H. Ahn was supported by the National Research Foundation of Korea (NRF) grant funded by the Korea government (MSIT) (2022R1C12007321).}
 
\thanks{$^*$ Corresponding author}

\newtheorem{theorem}{Theorem}[section]
\newtheorem{lemma}{Lemma}[section]
\newtheorem{corollary}{Corollary}[section]

\newtheorem{remark}{Remark}[section]

\newtheorem{definition}{Definition}[section]

\newcommand{\vast}{\bBigg@{4}}
\newcommand{\Vast}{\bBigg@{5}}

\begin{document}
	
	\date{\today}
	
\begin{abstract}
We introduce a logarithmic velocity alignment model for collective motion on
Riemannian manifolds. The interaction law is defined by the covariant time
derivative of logarithmic displacement vectors and therefore provides an
intrinsic geometric analogue of pairwise velocity-difference coupling. Under
two-sided nonpositive sectional-curvature bounds and a compatibility condition
between the interaction kernel and the curvature-induced growth of logarithmic
terms, we derive a kinetic-energy dissipation estimate and prove
interaction-weighted asymptotic velocity alignment, assuming that the
logarithmic interactions remain globally well-defined and that transported
velocity discrepancies have uniformly controlled time variation. In hyperbolic
space, both assumptions are verified directly from the geometry and the energy
estimates, yielding an unconditional interaction-weighted alignment result
within the stated class of global solutions.
\end{abstract}
	
	\maketitle

	%\tableofcontents

	\section{Introduction} \label{sec:1}
	\setcounter{equation}{0}
Collective dynamics arise in many interacting-agent systems, including
flocking, swarming, biological aggregation, synchronization, and opinion
dynamics. Representative models and analytical frameworks can be found in
\cite{C-S,G-P,H-K,H-K-S,Ku2,M-E,T-B-L,Wi1}. When the state space is a
Riemannian manifold, however, even the basic notion of a velocity difference
requires a geometric choice because velocities attached to different agents
belong to different tangent spaces.

This paper introduces and analyzes a logarithmic velocity alignment model on
Riemannian manifolds. Rather than comparing velocities at distinct tangent
spaces directly through a prescribed transport rule, the alignment mechanism is
defined by the covariant time variation of logarithmic displacement vectors
between agents. This choice is geometrically natural: logarithmic displacement vectors encode
intrinsic relative positions, while their covariant time derivatives encode
relative motion in the corresponding tangent spaces. Curvature therefore enters
the interaction law itself rather than only the subsequent estimates.
Before presenting the model, we fix the geometric notation used throughout.

Let $({\mathcal M},g)$ be a complete, connected, smooth $d$-dimensional
Riemannian manifold without boundary, and let $\nabla$ denote the Levi--Civita
connection on ${\mathcal M}$. For $x\in{\mathcal M}$, we write
$T_x{\mathcal M}$ for the tangent space at $x$ and $T{\mathcal M}$ for the
tangent bundle. If $v\in T_x{\mathcal M}$, then
\[
\|v\|_x:=\sqrt{g_x(v,v)}.
\]
The geodesic distance between two points $x,y\in{\mathcal M}$ is denoted by
$d(x,y)$.
We next recall the convention for parallel transport. Suppose that
$x,y\in{\mathcal M}$ are connected by a unique length-minimizing geodesic
\[
\eta_{xy}:[0,1]\to{\mathcal M},
\quad
\eta_{xy}(0)=y,
\quad
\eta_{xy}(1)=x.
\]
For $w\in T_y{\mathcal M}$, let $W$ be the vector field along $\eta_{xy}$
satisfying
\[
\nabla_{\dot\eta_{xy}(s)}W(s)=0,
\quad
W(0)=w.
\]
Then, the parallel transport of $w$ from $y$ to $x$ along $\eta_{xy}$ is defined
by
\[
P_{xy}w:=W(1),
\]
where
\[
P_{xy}:T_y{\mathcal M}\to T_x{\mathcal M}.
\]
Since parallel transport preserves the Riemannian metric, for
$w,w'\in T_y{\mathcal M}$, we have
\begin{equation*}
	P_{xy}P_{yx}=\mathrm{Id}_{T_x{\mathcal M}},
	\quad
	\|P_{xy}w\|_x=\|w\|_y,
	\quad
	g_x(P_{xy}w,P_{xy}w')=g_y(w,w'),
\end{equation*}
where $\mathrm{Id}_{T_x{\mathcal M}}$ is the identity map on
$T_x{\mathcal M}$.
We also use the exponential and logarithm maps. If $\gamma$ is the geodesic
satisfying
\[
\gamma(0)=x,
\quad
\dot\gamma(0)=v\in T_x{\mathcal M},
\]
then
\[
\exp_x v:=\gamma(1)
\]
whenever the right-hand side is defined. Conversely, if $y$ belongs to the
injectivity domain of $x$, then
\[
\log_x y:=\exp_x^{-1}(y)\in T_x{\mathcal M}.
\]
Equivalently, if $\gamma$ is the unique length-minimizing geodesic from $x$ to
$y$, then
\[
\log_x y=\dot\gamma(0).
\]
We denote by $\operatorname{Cut}(x)$ the cut locus of $x$. Thus, outside the cut
locus, the logarithm map based at $x$ is single-valued and smooth. In
particular, if $y\notin\operatorname{Cut}(x)$, then the geometric objects
$\log_x y$ and the parallel transport along the unique minimizing geodesic from
$x$ to $y$ are well-defined.
In what follows, the injectivity radius of ${\mathcal M}$ is denoted by
\[
\operatorname{inj}({\mathcal M})
:=
\inf_{x\in{\mathcal M}}\operatorname{inj}_x({\mathcal M}).
\]
Therefore, if
\[
d(x,y)<\operatorname{inj}({\mathcal M}),
\]
then $x$ and $y$ are connected by a unique length-minimizing geodesic, and both
$P_{xy}$ and $\log_x y$ are well-defined.
We will frequently employ the following elementary identities for logarithmic
vectors and parallel transport. Let $\tau\in(0,\infty]$, and let
$x_i(\cdot)$ and $x_j(\cdot)$ be $C^1$ curves on ${\mathcal M}$ defined on
$[0,\tau)$ with
\[
\dot x_i(t)=v_i(t)\in T_{x_i(t)}{\mathcal M},
\quad
\dot x_j(t)=v_j(t)\in T_{x_j(t)}{\mathcal M}.
\]
Assume that
\[
\sup_{0\le t<\tau}d(x_i(t),x_j(t))
<
\operatorname{inj}({\mathcal M}).
\]
Then, for all $t\in[0,\tau)$, the following relations hold:
\begin{align*}
P_{x_i x_j}\big(\log_{x_j}x_i\big)
&=-\log_{x_i}x_j,
&
\|\log_{x_i}x_j\|_{x_i}
&=d(x_i,x_j),
\\
\left|\frac{d}{dt}d(x_i,x_j)\right|
&\leq \|P_{x_i x_j}v_j-v_i\|_{x_i}.
\end{align*}
For standard background on Riemannian geometry and for further details on the
geometric identities stated above, we refer the reader to \cite{C1,Ju,P}.

With this notation in hand, we now introduce the logarithmic velocity alignment
model studied in this paper. For curves $x_i(\cdot)$ and $x_j(\cdot)$ on
$\mathcal M$ such that $\log_{x_i(t)}x_j(t)$ is well-defined, we use the
shorthand notation
\[
\nabla_{v_i}\log_{x_i}x_j
:=
\frac{D}{dt}\log_{x_i(t)}x_j(t),
\quad
v_i(t)=\dot x_i(t).
\]
Thus, $\nabla_{v_i}\log_{x_i}x_j$ denotes the covariant time derivative of the
logarithmic displacement vector along the curve $x_i(t)$. In particular, in the
Euclidean case, this quantity reduces to the usual relative velocity
$v_j-v_i$.

For the position--velocity pairs
$\{(x_i,v_i)\}_{i=1}^N$, we consider the following Cauchy problem:
\begin{equation}\label{Main}
	\begin{cases}
		\displaystyle
		\dot{x}_i=v_i,
		\quad t>0,
		\quad i\in[N]:=\{1,\ldots,N\},
		\\[0.2cm]
		\displaystyle
		\nabla_{v_i}v_i
		=
		\frac{\kappa}{N}
		\sum_{j=1}^{N}
		\phi(d(x_i,x_j))
		\nabla_{v_i}
		\log_{x_i}x_j,
		\\[0.2cm]
		\displaystyle
		(x_i(0),v_i(0))
		=
		(x_i^0,v_i^0)
		\in T\mathcal M .
	\end{cases}
\end{equation}
Here $N$ is the number of agents, $\kappa>0$ is the coupling strength of the
logarithmic velocity interaction, and
$\phi:[0,\infty)\to[0,\infty)$ is a Lipschitz continuous interaction
kernel.

The first equation in \eqref{Main} is the kinematic relation between position
and velocity. The second equation describes an alignment-type interaction
generated by the covariant time variation of logarithmic displacement vectors.
Since $\log_{x_i}x_j$ represents the intrinsic relative position of agent $j$
with respect to agent $i$, its covariant time derivative provides a geometric
notion of relative velocity measured from the viewpoint of agent $i$. Thus, each
agent adjusts its motion according to the change of the relative displacement
observed in its own tangent space. This gives an intrinsic formulation of a
velocity-alignment interaction through the logarithm map and the Levi--Civita
connection of the underlying manifold. In the Euclidean case, the logarithmic
displacement reduces to $x_j-x_i$, and its time derivative becomes the usual
velocity difference $v_j-v_i$. Hence, \eqref{Main} can be regarded as a natural
geometric extension of pairwise velocity-difference coupling to Riemannian
manifolds.

The analysis of \eqref{Main} is substantially different from the Euclidean
case. The term $\nabla_{v_i}\log_{x_i}x_j$ is determined by the variation of the
geodesic joining $x_i$ and $x_j$, and therefore reflects the
differential-geometric structure of the underlying manifold. This creates
several difficulties. The logarithm map may fail to be globally well-defined due
to the cut locus, velocities at different points lie in different tangent
spaces, and curvature affects both the sign and the strength of the energy
dissipation. In particular, nonpositive curvature yields a dissipative
structure, while positive curvature may destroy this monotonicity. Moreover, since the
size of the logarithmic velocity interaction may grow with the distance between
agents, the interaction kernel must be chosen carefully to control this
geometric growth. These features make it necessary to combine transported
velocity discrepancies with geometric estimates involving geodesic variations
and curvature.

\vspace{0.1cm}

We next introduce the notion of interaction-weighted velocity alignment.

\begin{definition}\label{D1.1}
	\emph{(Interaction-weighted asymptotic velocity alignment)}
	Let $\{(x_i(t),v_i(t))\}_{i=1}^N$ be a global solution to \eqref{Main} such that
	$P_{x_i(t)x_j(t)}$ is well-defined for all $i,j\in[N]$ and all $t\ge0$.
	We say that the solution exhibits interaction-weighted asymptotic velocity
	alignment if
	\[
	\lim_{t\to\infty}
	\max_{i,j\in[N]}
	\phi(d(x_i(t),x_j(t)))
	\|P_{x_i(t)x_j(t)}v_j(t)-v_i(t)\|_{x_i(t)}
	=
	0.
	\]
\end{definition}

The quantity
$
\|P_{x_i(t)x_j(t)}v_j(t)-v_i(t)\|_{x_i(t)}
$
measures the velocity discrepancy between agents $i$ and $j$ after transporting
$v_j(t)$ to the tangent space at $x_i(t)$. In Definition~\ref{D1.1}, this
discrepancy is weighted by the interaction strength
$\phi(d(x_i(t),x_j(t)))$. Thus, interaction-weighted velocity alignment means
that the velocity discrepancies vanish asymptotically in the
interaction-weighted sense naturally associated with the dissipation structure
of \eqref{Main}. In particular, this notion captures alignment among pairs that
continue to interact with non-negligible strength.
If, for a fixed pair $i,j\in[N]$, the interaction strength is bounded from
below along the corresponding trajectories, namely if there exists
$\phi_{ij,*}>0$ such that
\[
\phi(d(x_i(t),x_j(t)))\ge \phi_{ij,*}
\quad
\text{for all } t\ge0,
\]
then interaction-weighted asymptotic velocity alignment implies
\[
\lim_{t\to\infty}
\|P_{x_i(t)x_j(t)}v_j(t)-v_i(t)\|_{x_i(t)}
=
0
\]
for this pair. In particular, if such a lower bound holds for every pair $i,j\in[N]$, then,
since the number of agents is finite, interaction-weighted asymptotic velocity
alignment implies the usual asymptotic velocity alignment:
\[
\lim_{t\to\infty}
\max_{i,j\in[N]}
\|P_{x_i(t)x_j(t)}v_j(t)-v_i(t)\|_{x_i(t)}
=
0.
\]

\vspace{0.1cm}

Our main objective is therefore to identify geometric and interaction
conditions under which solutions of \eqref{Main} exhibit the
interaction-weighted asymptotic alignment introduced above.

\vspace{0.1cm}

The rest of this paper is organized as follows.
In Section~\ref{sec:2}, we recall several basic notions from Riemannian
geometry, including the Riemannian curvature tensor, sectional curvature, and
Jacobi fields. We also review Rauch comparison estimates and Barbalat's lemma,
which will be used in the subsequent analysis.
In Section~\ref{sec:3}, we establish the kinetic-energy dissipation estimate,
prove the general interaction-weighted alignment theorem under two \textit{a
priori} assumptions, and then verify those assumptions directly on hyperbolic
space. Section~\ref{sec:4} concludes the paper.

\section{Preliminaries} \label{sec:2}
\setcounter{equation}{0}
This section collects the geometric comparison estimates and the analytic
convergence lemma used below. In particular, Jacobi-field and Rauch comparison
arguments control the logarithmic interaction, while Barbalat's lemma converts
time-integrability of the dissipative quantity into pointwise convergence.

\subsection{Curvature, Jacobi fields, and comparison estimates}\label{sec:2.1}

\begin{definition} \label{D2.1}
	\emph{\cite{Ju,P}}~\emph{(Riemannian curvature
		tensor, sectional curvature, and Jacobi fields)}
	Let $(\mathcal M,g)$ be a Riemannian manifold, and let $\nabla$ be its
	Levi--Civita connection.
	\begin{enumerate}
		\item \emph{(Curvature tensor):}
		For smooth vector fields $X,Y,Z$ on $\mathcal M$, the Riemannian curvature
		tensor is given by
		\[
		R(X,Y)Z
		:=
		\nabla_X\nabla_Y Z
		-
		\nabla_Y\nabla_X Z
		-
		\nabla_{[X,Y]}Z.
		\]
		At each $x\in\mathcal M$, we use the following operator norm:
		\[
		\|R_x\|_{\mathrm{op}}
		:=
		\sup_{\substack{u,v,w\in T_x\mathcal M\\
				\|u\|_x=\|v\|_x=\|w\|_x=1}}
		\|R_x(u,v)w\|_x .
		\]
		
		\vspace{0.2cm}
		
		\item \emph{(Sectional curvature):}
		Let $x\in\mathcal M$ and let $u,w\in T_x\mathcal M$ be linearly
		independent. For the two-dimensional plane
		\[
		\sigma=\operatorname{span}\{u,w\}\subset T_x\mathcal M,
		\]
		the sectional curvature of $\sigma$ is defined by
		\[
		\sec_x(\sigma)
		:=
		\frac{
			g_x\bigl(R(u,w)w,u\bigr)
		}{
			g_x(u,u)g_x(w,w)-g_x(u,w)^2
		}.
		\]
		This value depends only on the plane $\sigma$, and not on the particular
		basis $u,w$.
		
		\vspace{0.2cm}
		
		\item \emph{(Jacobi fields):}
		Let $\gamma:[0,1]\to\mathcal M$ be a geodesic, where $s\in[0,1]$ denotes
		the geodesic parameter. A smooth vector field $J$ along $\gamma$ is called a
		Jacobi field if
		\[
		\nabla_s^2 J
		+
		R(J,\dot\gamma)\dot\gamma
		=
		0,
		\quad s\in[0,1].
		\]
		Here $\dot\gamma=d\gamma/ds$, $\nabla_s:=\nabla_{\dot\gamma}$, and
		$\nabla_s^2J:=\nabla_s(\nabla_sJ)$.
	\end{enumerate}
\end{definition}

The curvature tensor describes the noncommutativity of covariant differentiation.
The sectional curvature extracts from it a scalar curvature assigned to each
two-dimensional tangent plane. Jacobi fields describe infinitesimal deformations
of geodesics and will be used to express variations of logarithmic
displacement vectors.

\begin{definition}\label{D2.2}
	\emph{\cite{P}}~\emph{(Riemannian Hessian)}
	Let $f:\mathcal M\to\mathbb R$ be smooth. The Riemannian Hessian of $f$ at
	$p\in\mathcal M$ is the symmetric bilinear form
	\[
	\operatorname{Hess} f(p):T_p\mathcal M\times T_p\mathcal M\to\mathbb R
	\]
	defined by
	\[
	\operatorname{Hess} f(p)(u,w)
	:=
	g_p\bigl(\nabla_u\operatorname{grad}f,w\bigr),
	\quad
	u,w\in T_p\mathcal M.
	\]
\end{definition}

We next recall the comparison function, which will be employed to state the
comparison estimates in a compact form.

\begin{definition}\label{D2.3}
	\emph{\cite{P}}~\emph{(Comparison function)}
	For $\kappa\in\mathbb R$, define
	\[
	\operatorname{sn}_{\kappa}(r)
	:=
	\begin{cases}
		\displaystyle \frac{\sin(\sqrt{\kappa}\,r)}{\sqrt{\kappa}},
		& \kappa>0, \\[0.3cm]
		\displaystyle r,
		& \kappa=0, \\[0.3cm]
		\displaystyle \frac{\sinh(\sqrt{-\kappa}\,r)}{\sqrt{-\kappa}},
		& \kappa<0.
	\end{cases}
	\]
	Equivalently, $\operatorname{sn}_{\kappa}$ is the solution of
	\[
	y''+\kappa y=0,
	\quad
	y(0)=0,
	\quad
	y'(0)=1.
	\]
\end{definition}

The following comparison estimate gives upper and lower bounds for the Hessian
of the distance function under two-sided sectional curvature bounds.

\begin{lemma}\label{L2.1}
	\emph{\cite{P}}~\emph{(Hessian comparison)}
	Let $p\in\mathcal M$ and set
	\[
	r(q):=d(p,q).
	\]
	Assume that the sectional curvature of $(\mathcal M,g)$ satisfies
	\[
	\kappa_1\leq \operatorname{sec}\leq \kappa_2.
	\]
	Let $q\notin\{p\}\cup\operatorname{Cut}(p)$, and assume, if $\kappa_2>0$, that
	\[
	r(q)<\frac{\pi}{\sqrt{\kappa_2}}.
	\]
	Then, for every $w\in T_q\mathcal M$ with
	\[
	g_q(w,\operatorname{grad}r)=0,
	\]
	one has
	\[
	\frac{\operatorname{sn}_{\kappa_2}'(r(q))}
	{\operatorname{sn}_{\kappa_2}(r(q))}
	\|w\|_q^2
	\leq
	\operatorname{Hess} r(q)(w,w)
	\leq
	\frac{\operatorname{sn}_{\kappa_1}'(r(q))}
	{\operatorname{sn}_{\kappa_1}(r(q))}
	\|w\|_q^2 .
	\]
\end{lemma}

We will also use a Rauch-type estimate for the differential of the exponential
map. This estimate is useful when controlling differential quantities associated
with logarithmic vectors.

\begin{lemma}\label{L2.2}
	\emph{\cite{C1}}~\emph{(Rauch comparison for the exponential map)}
	Suppose that $\operatorname{sec}\leq K$. Let $x\in\mathcal M$ and
	$u\in T_x\mathcal M\setminus\{0\}$, and set
	\[
	r:=\|u\|_x<\operatorname{inj}_x\mathcal M.
	\]
	If $K>0$, assume in addition that
	\[
	r<\frac{\pi}{\sqrt K}.
	\]
	Then, for every $\xi\in T_x\mathcal M$ satisfying
	\[
	\langle \xi,u\rangle_x=0,
	\]
	we have
	\[
	\left\|d_u\exp_x(\xi)\right\|_{\exp_x(u)}
	\geq
	\frac{\operatorname{sn}_{K}(r)}{r}\|\xi\|_x.
	\]
	Moreover,
	\[
	\left\|d_u\exp_x\left(\frac{u}{r}\right)\right\|_{\exp_x(u)}=1.
	\]
	Consequently, for every $\xi\in T_x\mathcal M$,
	\[
	\left\|d_u\exp_x(\xi)\right\|_{\exp_x(u)}
	\geq
	\min\left\{1,\frac{\operatorname{sn}_{K}(r)}{r}\right\}\|\xi\|_x.
	\]
	The estimates at $u=0$ are understood by continuity.
\end{lemma}

The above estimates will be invoked later to control geometric terms generated
by the logarithmic velocity interaction in \eqref{Main}.

\subsection{An analytic convergence lemma}\label{sec:2.2}
We finish this section by recalling Barbalat's lemma. This result will be applied
to pass from time-integrability estimates to pointwise convergence of the
transported velocity discrepancies.

\begin{lemma}\label{L2.3}
	\emph{\cite{B}}~\emph{(Barbalat's lemma)}
	Let $f:[0,\infty)\to\mathbb R$ be uniformly continuous. If
	\[
	\int_0^\infty f(s)\,ds
	\]
	exists and is finite, then
	\[
	\lim_{t\to\infty}f(t)=0.
	\]
\end{lemma}

In the alignment analysis, Barbalat's lemma will be applied to quantities of the
form
$
\phi(d(x_i(t),x_j(t)))\|P_{x_i(t)x_j(t)}v_j(t)-v_i(t)\|_{x_i(t)}^2.
$
Once this quantity is shown to be integrable over $[0,\infty)$ and uniformly
continuous in time, Barbalat's lemma yields
\[
\lim_{t\to\infty}
\phi(d(x_i(t),x_j(t)))
\|P_{x_i(t)x_j(t)}v_j(t)-v_i(t)\|^2_{x_i(t)}
=0.
\]
Moreover, under the boundedness assumption on the interaction kernel imposed
below, this implies
\[
\lim_{t\to\infty}
\phi(d(x_i(t),x_j(t)))
\|P_{x_i(t)x_j(t)}v_j(t)-v_i(t)\|_{x_i(t)}
=0.
\]
Thus, Barbalat's lemma will be used to obtain the interaction-weighted velocity
alignment result in Definition~\ref{D1.1}.

\section{Main results} \label{sec:3}
\setcounter{equation}{0}
We now establish the asymptotic alignment results for \eqref{Main}. We first
state the geometric and interaction assumptions, then derive the energy
dissipation mechanism and the general alignment theorem. Finally, in hyperbolic
space, we verify the two auxiliary \textit{a priori} assumptions directly.

\subsection{Sufficient conditions} \label{sec:3.1}
In this subsection, we formulate the sufficient conditions used in the main
alignment theorem. 

\vspace{0.2cm}

\begin{itemize}
	\item $(\mathcal F_1)$
	{(Underlying geometry):}
	Let $({\mathcal M},g)$ be a complete, connected, smooth
	$d$-dimensional Riemannian manifold without boundary. We assume that there
	exists a constant $k\le 0$ such that the sectional curvature of
	$({\mathcal M},g)$ satisfies
	\[
	k \le \sec \le 0.
	\]
	\item $(\mathcal F_2)$
	{(Interaction kernel and coupling strength):}
	The coupling strength satisfies
	\[
	\kappa>0,
	\]
	and the interaction kernel
	$
	\phi:[0,\infty)\to[0,\infty)
	$
	is Lipschitz continuous. Moreover, we assume that
	\[
	M
	:=
	\sup_{r\ge0}
	\Big(
	\Theta_k(r)\phi(r)
	\Big)
	<
	\infty,
	\]
	where
	\[
	\Theta_k(r)
	:=
	\begin{cases}
		\sqrt{|k|}\,r\coth(\sqrt{|k|}\,r),
		& k<0, \\[0.2cm]
		1,
		& k=0.
	\end{cases}
	\]
	Here $\Theta_k(0)$ is understood by continuity.
\end{itemize}

\begin{remark} \label{R3.1}
	\emph{(Comments on $(\mathcal F_1)-(\mathcal F_2)$)}
	The assumptions $(\mathcal F_1)-(\mathcal F_2)$ are imposed for the following
	reasons.
	\begin{enumerate}
		\item
		The assumption $(\mathcal F_1)$ is the main geometric condition behind the
		energy dissipation of the system \eqref{Main}. The nonpositivity of the
		sectional curvature guarantees that the index-form contribution arising from
		the logarithmic velocity interaction has the correct sign. As a consequence,
		the kinetic energy is nonincreasing and one obtains uniform-in-time energy
		estimates. These estimates provide the basic starting point for the long-time
		analysis of the alignment dynamics. If positive sectional curvature is
		allowed, this dissipative structure may break down, and the kinetic energy
		may increase in general. The lower curvature bound $k\le \sec$ is used to
		obtain quantitative comparison estimates for the logarithmic terms.
		
		\vspace{0.2cm}
		
		\item
		The condition $\kappa>0$ in $(\mathcal F_2)$ means that the interaction acts
		in the dissipative direction. The Lipschitz continuity of $\phi$ is imposed to ensure the well-posedness of
		the Cauchy problem \eqref{Main} and to obtain the regularity needed for the
		Barbalat-type argument in the long-time analysis. The bound
		\[
		\sup_{r\ge0}\Theta_k(r)\phi(r)<\infty
		\]
		is used to control the geometric growth of the logarithmic velocity
		interaction. 
		Indeed, under the curvature bound $k\le \sec\le0$, the
		Rauch-type comparison estimates in Lemma \ref{L2.1} and Lemma \ref{L2.2} give
		coefficients of order $\Theta_k(r)$ in the estimates of the logarithmic
		terms $\nabla_{v_i}\log_{x_i}x_j$. Hence, the above bound, together with the uniform speed bound obtained
		from the energy estimate, allows us to control the weighted interaction terms $\phi(d(x_i,x_j))
		\nabla_{v_i}\log_{x_i}x_j$	appearing in the system \eqref{Main}. This control is later employed in the Barbalat-type
		argument for the interaction-weighted alignment result.
	\end{enumerate}
\end{remark}

\subsection{Interaction-weighted asymptotic velocity alignment}
\label{sec:3.2}

In this subsection, we study interaction-weighted asymptotic velocity alignment
under the following two \textit{a priori} assumptions. Throughout this subsection,
let $\{(x_i(t),v_i(t))\}_{i=1}^N$ be a global solution to \eqref{Main}.

\begin{itemize}
	\item $(\mathcal P_1)$
	{(Well-defined logarithmic interactions):}
	For every $i,j\in[N]$ and $t\ge0$, the logarithmic vector
	$\log_{x_i(t)}x_j(t)$ and the parallel transport map
	$P_{x_i(t)x_j(t)}$ are well-defined. Equivalently, we assume that
	\[
	x_j(t)\notin \operatorname{Cut}(x_i(t))
	\quad
	\text{for all } i,j\in[N] \text{ and } t\ge0.
	\]
	
	\item $(\mathcal P_2)$
{(Uniform time regularity of transported velocity discrepancies):}
	There exists a constant $C>0$ such that
	\[
\sup_{t\ge 0}\max_{i,j\in[N]}
\left|
\frac{d}{dt}
\|P_{x_i(t)x_j(t)}v_j(t)-v_i(t)\|_{x_i(t)}^2
\right|
\le C.
	\]
\end{itemize}

The assumption $(\mathcal P_1)$ ensures that all logarithmic interactions and
transported velocity discrepancies appearing in the system are well-defined.
The assumption $(\mathcal P_2)$ will be used later to obtain the uniform
continuity needed in the Barbalat-type argument. These two assumptions are imposed only in the general Riemannian manifold setting. In the specific case where the underlying manifold is the
$d$-dimensional hyperbolic space $\mathbb H^d$, realized as the hyperboloid
model, they will be removed in Section~\ref{sec:3.3} by verifying them directly
from the geometry of $\mathbb H^d$.

\vspace{0.1cm}

Now, we define the kinetic energy of the system by
\[
\mathcal E(t)
:=
\frac12
\sum_{i=1}^N
\|v_i(t)\|_{x_i(t)}^2 .
\]
We first prove that this energy is nonincreasing along solutions.

\begin{lemma} \label{L3.1} \emph{(Kinetic energy dissipation)} 
Let $\{(x_i(t),v_i(t))\}_{i=1}^N$ be a global solution to \eqref{Main}.
Assume that the sufficient conditions $(\mathcal F_1)-(\mathcal F_2)$ hold
and that the \textit{a priori} assumption $(\mathcal P_1)$ holds for the
system \eqref{Main}. Then the kinetic energy $\mathcal E(t)$ satisfies
\[
\frac{d}{dt}\mathcal E(t)
=
-\frac{\kappa}{2N}
\sum_{i,j=1}^N
\phi(d(x_i(t),x_j(t)))
I_{ij}(J_{ij},J_{ij}),
\]
and hence
\[
\frac{d}{dt}\mathcal E(t)
\le
-\frac{\kappa}{2N}
\sum_{i,j=1}^N
\phi(d(x_i(t),x_j(t)))
\|P_{x_i(t)x_j(t)}v_j(t)-v_i(t)\|_{x_i(t)}^2
\le 0.
\]
Here, for $i\ne j$, $J_{ij}$ denotes the Jacobi field along the geodesic
$\gamma_{ij}$ joining $x_i(t)$ and $x_j(t)$ with
\[
J_{ij}(0)=v_i(t),
\quad
J_{ij}(1)=v_j(t),
\]
and $I_{ij}$ denotes the corresponding index form, defined by
\[
I_{ij}(X,Y)
:=
\int_0^1
\left(
\left\langle \nabla_s X,\nabla_s Y\right\rangle_{\gamma_{ij}(s)}
-
\left\langle
R(X,\dot\gamma_{ij})\dot\gamma_{ij},
Y
\right\rangle_{\gamma_{ij}(s)}
\right)\,ds
\]
for vector fields $X,Y$ along $\gamma_{ij}$. The diagonal terms $i=j$ are
understood to be zero. In particular, $\mathcal E(t)$ is nonincreasing in time.
\end{lemma}

\begin{proof}
	By the definition of $\mathcal E(t)$ and the system \eqref{Main}, we have
	\[
	\begin{aligned}
		\frac{d}{dt}\mathcal E(t)
		=
		\sum_{i=1}^N
		\left\langle \nabla_{v_i}v_i,v_i\right\rangle_{x_i}
		=
		\frac{\kappa}{N}
		\sum_{i,j=1}^N
		\phi(d(x_i,x_j))
		\left\langle
		\nabla_{v_i}\log_{x_i}x_j,
		v_i
		\right\rangle_{x_i}.
	\end{aligned}
	\]
	Symmetrizing the double sum gives
	\begin{align}\label{C0}
	\begin{aligned}
		\frac{d}{dt}\mathcal E(t)
		&=
		\frac{\kappa}{2N}
		\sum_{i,j=1}^N
		\phi(d(x_i,x_j))
		\Big(
		\left\langle
		\nabla_{v_i}\log_{x_i}x_j,
		v_i
		\right\rangle_{x_i}
		+
		\left\langle
		\nabla_{v_j}\log_{x_j}x_i,
		v_j
		\right\rangle_{x_j}
		\Big).
	\end{aligned}
	\end{align}
For each $i\ne j$, let $\Gamma_{ij}(t,s)$ be the geodesic variation generated
by the moving endpoints $x_i(t)$ and $x_j(t)$, that is,
\[
\Gamma_{ij}(t,0)=x_i(t),
\quad
\Gamma_{ij}(t,1)=x_j(t).
\]
For fixed $t$, we write
\[
\gamma_{ij}(s):=\Gamma_{ij}(t,s),
\quad
J_{ij}(s):=\partial_t\Gamma_{ij}(t,s).
\]
Then, $J_{ij}$ is a Jacobi field along $\gamma_{ij}$ and satisfies
\[
J_{ij}(0)=v_i,
\quad
J_{ij}(1)=v_j.
\]
Since
\[
\log_{x_i}x_j=\partial_s\Gamma_{ij}(t,0),
\]
the torsion-free property of the Levi--Civita connection gives
\[
\nabla_{v_i}\log_{x_i}x_j
=
\frac{D}{dt}\partial_s\Gamma_{ij}(t,0)
=
\frac{D}{ds}\partial_t\Gamma_{ij}(t,0)
=
\nabla_sJ_{ij}(0).
\]
Similarly, since
\[
\log_{x_j}x_i=-\partial_s\Gamma_{ij}(t,1),
\]
we have
\[
\nabla_{v_j}\log_{x_j}x_i
=
-\frac{D}{dt}\partial_s\Gamma_{ij}(t,1)
=
-\frac{D}{ds}\partial_t\Gamma_{ij}(t,1)
=
-\nabla_sJ_{ij}(1).
\]
Hence, using
\[
J_{ij}(0)=v_i,
\quad
J_{ij}(1)=v_j,
\]
we obtain
\[
\begin{aligned}
	&
	\left\langle
	\nabla_{v_i}\log_{x_i}x_j,
	v_i
	\right\rangle_{x_i}
	+
	\left\langle
	\nabla_{v_j}\log_{x_j}x_i,
	v_j
	\right\rangle_{x_j}
	\\
	&\quad =
	\left\langle
	\nabla_s J_{ij}(0),
	J_{ij}(0)
	\right\rangle_{x_i}
	-
	\left\langle
	\nabla_s J_{ij}(1),
	J_{ij}(1)
	\right\rangle_{x_j}
	\\
	&\quad =
	-
	\left[
	\left\langle
	\nabla_s J_{ij},
	J_{ij}
	\right\rangle_{\gamma_{ij}}
	\right]_{s=0}^{s=1}
	\\
	&\quad =
	-
	\int_0^1
	\frac{d}{ds}
	\left\langle
	\nabla_s J_{ij},
	J_{ij}
	\right\rangle_{\gamma_{ij}(s)}
	\,ds
	\\
	&\quad =
	-
	\int_0^1
	\left(
	\|\nabla_s J_{ij}\|_{\gamma_{ij}(s)}^2
	+
	\left\langle
	\nabla_s^2 J_{ij},
	J_{ij}
	\right\rangle_{\gamma_{ij}(s)}
	\right)
	\,ds.
\end{aligned}
\]
Since $J_{ij}$ is a Jacobi field along $\gamma_{ij}$, it satisfies
\[
\nabla_s^2 J_{ij}
+
R(J_{ij},\dot\gamma_{ij})\dot\gamma_{ij}
=
0.
\]
Thus,
\[
\begin{aligned}
	&
	\left\langle
	\nabla_{v_i}\log_{x_i}x_j,
	v_i
	\right\rangle_{x_i}
	+
	\left\langle
	\nabla_{v_j}\log_{x_j}x_i,
	v_j
	\right\rangle_{x_j}
	\\
	&\quad =
	-
	\int_0^1
	\left(
	\|\nabla_s J_{ij}\|_{\gamma_{ij}(s)}^2
	-
	\left\langle
	R(J_{ij},\dot\gamma_{ij})\dot\gamma_{ij},
	J_{ij}
	\right\rangle_{\gamma_{ij}(s)}
	\right)
	\,ds
	\\
	&\quad =
	-
	I_{ij}(J_{ij},J_{ij}).
\end{aligned}
\]
Therefore, we obtain
\[
\left\langle
\nabla_{v_i}\log_{x_i}x_j,
v_i
\right\rangle_{x_i}
+
\left\langle
\nabla_{v_j}\log_{x_j}x_i,
v_j
\right\rangle_{x_j}
=
-
I_{ij}(J_{ij},J_{ij}).
\]
	Substituting this identity into \eqref{C0} gives
	\[
	\frac{d}{dt}\mathcal E(t)
	=
	-\frac{\kappa}{2N}
	\sum_{i,j=1}^{N}
	\phi(d(x_i,x_j))
	I_{ij}(J_{ij},J_{ij}).
	\]
	Then, it remains to estimate the index form from below. Since $\sec\le0$ by
	$(\mathcal F_1)$, one has
	\[
	I_{ij}(J_{ij},J_{ij})
	\ge
	\int_0^1
	\|\nabla_s J_{ij}(s)\|_{\gamma_{ij}(s)}^2\,ds.
	\]
	Now, we set
	\[
	\widetilde J_{ij}(s)
	:=
	P_{x_i\gamma_{ij}(s)}J_{ij}(s).
	\]
	Since parallel transport preserves norms, we obtain
	\[
	\int_0^1
	\|\nabla_s J_{ij}(s)\|_{\gamma_{ij}(s)}^2\,ds
	=
	\int_0^1
	\left\|
	\frac{d}{ds}\widetilde J_{ij}(s)
	\right\|_{x_i}^2
	\,ds.
	\]
	By the Cauchy--Schwarz inequality,
	\[
	\int_0^1
	\left\|
	\frac{d}{ds}\widetilde J_{ij}(s)
	\right\|_{x_i}^2
	\,ds
	\ge
	\left\|
	\widetilde J_{ij}(1)-\widetilde J_{ij}(0)
	\right\|_{x_i}^2.
	\]
	Since
	\[
	\widetilde J_{ij}(0)=v_i,
	\quad
	\widetilde J_{ij}(1)=P_{x_i x_j}v_j,
	\]
	we get
	\[
	I_{ij}(J_{ij},J_{ij})
	\ge
	\|P_{x_i x_j}v_j-v_i\|_{x_i}^2.
	\]
	Consequently,
	\[
	\frac{d}{dt}\mathcal E(t)
	\le
	-\frac{\kappa}{2N}
	\sum_{i,j=1}^{N}
	\phi(d(x_i,x_j))
	\|P_{x_i x_j}v_j-v_i\|_{x_i}^2
	\le0.
	\]
	This proves the energy dissipation estimate.
\end{proof}

The dissipation estimate immediately yields a uniform velocity bound and
integrability in time of the weighted squared velocity discrepancies.

\begin{corollary}\label{C3.1}\emph{(Energy estimates)}
	Let $\{(x_i(t),v_i(t))\}_{i=1}^N$ be a global solution to the system
	\eqref{Main}. Assume that the sufficient conditions
	$(\mathcal F_1)-(\mathcal F_2)$ hold and that the \textit{a priori}
	assumption $(\mathcal P_1)$ holds for the system \eqref{Main}. Then, for all
	$t\ge0$, we have
	\[
	\max_{i\in[N]}
	\|v_i(t)\|_{x_i(t)}
	\le
	\sqrt{2\mathcal E(0)}.
	\]
	Moreover, one has
	\[
	\sum_{i,j=1}^N\int_0^\infty
	\left(
	\phi(d(x_i(s),x_j(s)))
	\|P_{x_i(s)x_j(s)}v_j(s)-v_i(s)\|_{x_i(s)}^2
	\right)
	\,ds
	\le
	\frac{2N}{\kappa}\mathcal E(0)
	<
	\infty.
	\]
\end{corollary}

\begin{proof}
	Lemma~\ref{L3.1} gives
	\[
	\mathcal E(t)\le \mathcal E(0)
	\quad
	\text{for all } t\ge0.
	\]
	Hence, for each $i\in[N]$,
	\[
	\|v_i(t)\|_{x_i(t)}
	\le
	\sqrt{2\mathcal E(t)}
	\le
	\sqrt{2\mathcal E(0)},
	\]
which yields
	\[
	\max_{i\in[N]}
	\|v_i(t)\|_{x_i(t)}
	\le
	\sqrt{2\mathcal E(0)}.
	\]
	Next, by integrating the energy dissipation estimate in Lemma~\ref{L3.1} over
	$[0,t]$, we get
	\[
	\frac{\kappa}{2N}
	\sum_{i,j=1}^N
	\int_0^t
	\phi(d(x_i(s),x_j(s)))
	\|P_{x_i(s)x_j(s)}v_j(s)-v_i(s)\|_{x_i(s)}^2
	\,ds
	\le
	\mathcal E(0)-\mathcal E(t)\leq \mathcal E(0).
	\]
	Then, it follows that
	\[
		\sum_{i,j=1}^N
	\int_0^t
	\phi(d(x_i(s),x_j(s)))
	\|P_{x_i(s)x_j(s)}v_j(s)-v_i(s)\|_{x_i(s)}^2
	\,ds
	\le
	\frac{2N}{\kappa}\mathcal E(0)<\infty.
	\]
	Letting $t\to\infty$ proves the second assertion.
\end{proof}

We are now ready to prove the main alignment result of this subsection.
By combining the energy dissipation estimate with Barbalat's lemma
(Lemma~\ref{L2.3}), we show that interaction-weighted asymptotic velocity
alignment holds under the sufficient conditions $(\mathcal F_1)-(\mathcal F_2)$
and the \textit{a priori} assumptions $(\mathcal P_1)-(\mathcal P_2)$.

\begin{theorem}\label{T3.1}~\emph{(Interaction-weighted asymptotic velocity alignment)}
	Let $\{(x_i(t),v_i(t))\}_{i=1}^N$ be a global solution to the system
	\eqref{Main}. Assume that the sufficient conditions
	$(\mathcal F_1)-(\mathcal F_2)$ hold and that the \textit{a priori}
	assumptions $(\mathcal P_1)-(\mathcal P_2)$ are satisfied.
	Then, the solution exhibits interaction-weighted asymptotic velocity
	alignment. More precisely, one has
	\[
	\lim_{t\to\infty}
	\max_{i,j\in[N]}
	\phi(d(x_i(t),x_j(t)))
	\,
	\|P_{x_i(t)x_j(t)}v_j(t)-v_i(t)\|_{x_i(t)}
	=
	0.
	\]
\end{theorem}

\begin{proof}
Since the maximum of finitely many nonnegative terms is bounded by their sum, Corollary~\ref{C3.1} yields
\[
\int_0^\infty
\max_{i,j\in[N]}
\left(
\phi(d(x_i(t),x_j(t)))
\|P_{x_i(t)x_j(t)}v_j(t)-v_i(t)\|_{x_i(t)}^2
\right)
dt
<
\infty.
\]
Next, we show that the integrand is uniformly continuous. First, by
Corollary~\ref{C3.1}, the velocities are uniformly bounded. Hence, for any
$s,t\ge0$ and $i,j\in[N]$, we have
\[
\begin{aligned}
	&
	\left|
	d(x_i(t),x_j(t))-d(x_i(s),x_j(s))
	\right|
	\\
	&\quad \le
	d(x_i(t),x_i(s))+d(x_j(t),x_j(s))
	\\
	&\quad \le
	\int_s^t
	\left(
	\|v_i(\tau)\|_{x_i(\tau)}
	+
	\|v_j(\tau)\|_{x_j(\tau)}
	\right)
	d\tau
	\\
	&\quad \le
	2\sqrt{2\mathcal E(0)}\,|t-s|.
\end{aligned}
\]
Since $\phi$ is Lipschitz continuous by $(\mathcal F_2)$, it follows
that
\[
t
\mapsto
\phi(d(x_i(t),x_j(t)))
\]
is uniformly continuous on $[0,\infty)$ for every pair $i,j\in[N]$.
On the other hand, the \textit{a priori} assumption $(\mathcal P_2)$ implies
that
\[
t
\mapsto
\|P_{x_i(t)x_j(t)}v_j(t)-v_i(t)\|_{x_i(t)}^2
\]
is uniformly continuous on $[0,\infty)$ for every pair $i,j\in[N]$.
Moreover, both factors are uniformly bounded. Indeed, by $(\mathcal F_2)$,
the interaction kernel $\phi$ is bounded, and by Corollary~\ref{C3.1},
\[
\|P_{x_i(t)x_j(t)}v_j(t)-v_i(t)\|_{x_i(t)}
\le
\|v_j(t)\|_{x_j(t)}
+
\|v_i(t)\|_{x_i(t)}
\le
2\sqrt{2\mathcal E(0)}.
\]
Therefore, the product
\[
t
\mapsto
\phi(d(x_i(t),x_j(t)))
\|P_{x_i(t)x_j(t)}v_j(t)-v_i(t)\|_{x_i(t)}^2
\]
is uniformly continuous on $[0,\infty)$ for every pair $i,j\in[N]$.
Since the maximum is taken over finitely many pairs, the function
\[
t
\mapsto
\max_{i,j\in[N]}
\left(
\phi(d(x_i(t),x_j(t)))
\|P_{x_i(t)x_j(t)}v_j(t)-v_i(t)\|_{x_i(t)}^2
\right)
\]
is also uniformly continuous on $[0,\infty)$.
Therefore, Barbalat's lemma (Lemma~\ref{L2.3}) implies
\[
\lim_{t\to\infty}
\max_{i,j\in[N]}
\phi(d(x_i(t),x_j(t)))
\|P_{x_i(t)x_j(t)}v_j(t)-v_i(t)\|_{x_i(t)}^2
=
0.
\]
Finally, since $\phi$ is bounded by $(\mathcal F_2)$, we have
\begin{align*}
0
&\le
\phi(d(x_i(t),x_j(t)))
\|P_{x_i(t)x_j(t)}v_j(t)-v_i(t)\|_{x_i(t)}
\\
&\le
\sqrt{\|\phi\|_{L^\infty}}
\left(
\phi(d(x_i(t),x_j(t)))
\|P_{x_i(t)x_j(t)}v_j(t)-v_i(t)\|_{x_i(t)}^2
\right)^{1/2}.
\end{align*}
Taking the maximum over $i,j\in[N]$ and letting $t\to\infty$, we obtain
\[
\lim_{t\to\infty}
\max_{i,j\in[N]}
\phi(d(x_i(t),x_j(t)))
\|P_{x_i(t)x_j(t)}v_j(t)-v_i(t)\|_{x_i(t)}
=
0.
\]
This completes the proof.
\end{proof}

We conclude this subsection by recording a uniform estimate for the weighted
logarithmic interactions. The estimate holds in the general Riemannian setting
under $(\mathcal F_1)-(\mathcal F_2)$ and $(\mathcal P_1)$, and it will also
be used in the hyperbolic-space analysis below.

\begin{lemma}\label{L3.2}\emph{(Uniform bound for weighted logarithmic interactions)}
	Let $\{(x_i(t),v_i(t))\}_{i=1}^N$ be a global solution to the system
	\eqref{Main}. Assume that the sufficient conditions
	$(\mathcal F_1)-(\mathcal F_2)$ hold and that the \textit{a priori}
	assumption $(\mathcal P_1)$ holds for the system \eqref{Main}. Then, we have
	\[
	\sup_{t\geq 0}\max_{i,j\in[N]}
	\phi(d(x_i(t),x_j(t)))
	\left\|
	\nabla_{v_i}\log_{x_i(t)}x_j(t)
	\right\|_{x_i(t)}
	\le
	2M\sqrt{2\mathcal E(0)},
	\]
	where $M$ is the constant in $(\mathcal F_2)$.
	In particular, one has
	\[
	\sup_{t\geq 0}\max_{i\in[N]}
	\|\nabla_{v_i}v_i(t)\|_{x_i(t)}
	\le
	2\kappa M\sqrt{2\mathcal E(0)}.
	\]
\end{lemma}

\begin{proof}
	Fix $i,j\in[N]$ and $t\ge0$. Since the case $i=j$ is trivial, we assume
	$i\ne j$. For this fixed pair and time, we write
	\[
	x:=x_i(t),
	\quad
	y:=x_j(t),
	\quad
	u:=v_i(t)\in T_x\mathcal M,
	\quad
	w:=v_j(t)\in T_y\mathcal M,
	\]
	and
	\[
	\eta:=\log_x y,
	\quad
	r:=\|\eta\|_x=d(x,y).
	\]
	By $(\mathcal P_1)$, the logarithmic vector $\log_x y$ is well-defined, and
	$d_\eta\exp_x:T_x\mathcal M\to T_y\mathcal M$ is invertible.
	Let
	\[
	F_y(z):=\frac12 d(z,y)^2.
	\]
	Since
	\[
	\operatorname{grad}F_y(x)=-\log_x y,
	\]
	differentiation with respect to the first variable gives
	\[
	-\operatorname{Hess}F_y(x)[u].
	\]
If the first variable $x$ is fixed and the second variable $y$ varies in the
direction $w\in T_y\mathcal M$, differentiating
\[
\exp_x(\log_x y)=y
\]
gives
\[
d_\eta\exp_x
\left[
D_y(\log_x)(w)
\right]
=
w,
\quad
\eta=\log_x y.
\]
Hence,
\[
D_y(\log_x)(w)
=
\bigl(d_\eta\exp_x\bigr)^{-1}w.
\]
	Therefore, the covariant chain rule for the map $(x,y)\mapsto\log_x y$ yields
	\[
	\nabla_{v_i}\log_{x_i}x_j
	=
	-\operatorname{Hess}F_y(x)[u]
	+
	\bigl(d_\eta\exp_x\bigr)^{-1}w.
	\]
	Now, we first estimate the Hessian term. By Lemma~\ref{L2.1} and $(\mathcal F_1)-(\mathcal F_2)$, the Hessian of $F_y$ satisfies
	\begin{align}\label{C1}
	\left\|
	\operatorname{Hess}F_y(x)[u]
	\right\|_x
	\le
	\Theta_k(r)\|u\|_x,
	\end{align}
	where $\Theta_k$ is the comparison coefficient defined in $(\mathcal F_2)$.
	We next estimate the inverse exponential-map term. By Lemma~\ref{L2.2} and 	$(\mathcal F_1)$, we get
	\[
	\|d_\eta\exp_x(\xi)\|_y
	\ge
	\|\xi\|_x
	\quad
	\text{for all } \xi\in T_x\mathcal M.
	\]
	Thus,
	\[
	\left\|
	\bigl(d_\eta\exp_x\bigr)^{-1}w
	\right\|_x
	\le
	\|w\|_y.
	\]
	Since $\Theta_k(r)\ge1$, we obtain
	\begin{align}\label{C2}
	\left\|
	\bigl(d_\eta\exp_x\bigr)^{-1}w
	\right\|_x
	\le
	\Theta_k(r)\|w\|_y.
	\end{align}
	Combining \eqref{C1} and \eqref{C2}, it follows that
	\[
	\left\|
	\nabla_{v_i}\log_{x_i}x_j
	\right\|_x
	\le
	\Theta_k(r)
	\left(
	\|u\|_x+\|w\|_y
	\right),
	\]
	and then we have
	\[
	\left\|
	\nabla_{v_i}\log_{x_i(t)}x_j(t)
	\right\|_{x_i(t)}
	\le
	\Theta_k(d(x_i(t),x_j(t)))
	\left(
	\|v_i(t)\|_{x_i(t)}
	+
	\|v_j(t)\|_{x_j(t)}
	\right).
	\]
	Multiplying both sides by $\phi(d(x_i(t),x_j(t)))$ and using the definition of
	$M$ in $(\mathcal F_2)$, we deduce
	\[
	\phi(d(x_i(t),x_j(t)))
	\left\|
	\nabla_{v_i}\log_{x_i(t)}x_j(t)
	\right\|_{x_i(t)}
	\le
	M
	\left(
	\|v_i(t)\|_{x_i(t)}
	+
	\|v_j(t)\|_{x_j(t)}
	\right).
	\]
	Then, employing Corollary~\ref{C3.1}, we estimate
	\[
	\phi(d(x_i(t),x_j(t)))
	\left\|
	\nabla_{v_i}\log_{x_i(t)}x_j(t)
	\right\|_{x_i(t)}
	\le
	2M\sqrt{2\mathcal E(0)}.
	\]
	Taking the maximum over $i,j\in[N]$ gives the desired first assertion.
Moreover, the equation of motion gives
	\[
	\|\nabla_{v_i}v_i(t)\|_{x_i(t)}
	\le
	\frac{\kappa}{N}
	\sum_{j=1}^{N}
	\phi(d(x_i(t),x_j(t)))
	\left\|
	\nabla_{v_i}\log_{x_i(t)}x_j(t)
	\right\|_{x_i(t)}
	\le
	2\kappa M\sqrt{2\mathcal E(0)}.
	\]
	Consequently, taking the maximum over $i\in[N]$ completes the proof.
\end{proof}

\subsection{Removal of the {a priori} assumptions}
\label{sec:3.3}
We now specialize to the $d$-dimensional hyperbolic space $\mathbb H^d$,
realized as the hyperboloid model, and verify the \textit{a priori} assumptions
$(\mathcal P_1)$ and $(\mathcal P_2)$ directly. We begin by recalling the
basic geometric formulas for the hyperboloid model, including the tangent
space, sectional curvature, covariant derivative, geodesic distance,
exponential and logarithm maps, injectivity radius, parallel transport, and
the time derivative of the transported velocity discrepancy.

\begin{lemma}
	\label{L3.3}
	\emph{\cite{A-B-H-Y,A-H-K-S}}
	\emph{(Basic geometric formulas on the hyperboloid)}
	Let
	\[
	\mathbb H^d
	:=
	\left\{
	x\in\mathbb R^{d+1}
	:
	\langle x,x\rangle_L=-1,
	\hspace{0.2cm}
	x^0>0
	\right\},
	\]
	where
	\[
	\langle x,y\rangle_L
	:=
	-x^0y^0+\sum_{\alpha=1}^d x^\alpha y^\alpha
	\]
	denotes the Lorentzian inner product, also called the Minkowski inner product. Then, the following properties hold.
	
	\begin{enumerate}
	\item
	For each $x\in\mathbb H^d$, the tangent space is given by
	\[
	T_x\mathbb H^d
	=
	\left\{
	v\in\mathbb R^{d+1}
	:
	\langle x,v\rangle_L=0
	\right\}.
	\]
	The Riemannian metric on $\mathbb H^d$ is the restriction of
	$\langle\cdot,\cdot\rangle_L$ to $T_x\mathbb H^d$. With this metric,
	$\mathbb H^d$ is a complete Riemannian manifold with constant sectional
	curvature
	\[
	\sec \equiv -1.
	\]
		\item
	Let $x(t)$ be a smooth curve on $\mathbb H^d$, and let $Z(t)$ be a smooth
	vector field along $x(t)$. Then the covariant derivative is given by
	\[
	\nabla_{\dot x}Z
	=
	\dot Z
	-
	\langle \dot x,Z\rangle_L x.
	\]
		\item
		For every $x,y\in\mathbb H^d$, the geodesic distance satisfies
		\[
		\cosh d_{\mathbb H}(x,y)
		=
		-\langle x,y\rangle_L.
		\]
		In particular, the injectivity radius of $\mathbb H^d$ is infinite:
		\[
		\operatorname{inj}_x\mathbb H^d=\infty,
		\hspace{0.2cm}
		\text{for all } x\in\mathbb H^d.
		\]
		\item
		For $x\in\mathbb H^d$ and $v\in T_x\mathbb H^d$, the exponential map is
		given by
		\[
		\exp_x(v)
		=
		\cosh(\|v\|_x)x
		+
		\sinh(\|v\|_x)\frac{v}{\|v\|_x},
		\quad v\neq 0.
		\]
		\item
		For $x,y\in\mathbb H^d$ with $x\ne y$, the logarithm map is given by
		\[
		\log_x y
		=
		\frac{d_{\mathbb H}(x,y)}
		{\sqrt{\langle x,y\rangle_L^2-1}}
		\left(
		y+\langle x,y\rangle_L x
		\right).
		\]
		\item
		For $x,y\in\mathbb H^d$, the parallel transport from
		$T_y\mathbb H^d$ to $T_x\mathbb H^d$ along the unique geodesic joining
		$y$ to $x$ is given by
		\[
		P_{xy}v
		=
		v
		+
		\frac{\langle x,v\rangle_L}
		{1-\langle x,y\rangle_L}
		(x+y),
		\quad
		v\in T_y\mathbb H^d.
		\]
		
		\item
		Let $x_i(t)$ and $x_j(t)$ be smooth curves on $\mathbb H^d$, and let
		\[
		v_i(t)=\dot x_i(t),
		\quad
		v_j(t)=\dot x_j(t).
		\]
		Then, the transported velocity discrepancy satisfies the identity
		\[
		\begin{aligned}
			&
			\frac{d}{dt}
			\|P_{x_i(t)x_j(t)}v_j(t)-v_i(t)\|_{x_i(t)}^2
			\\
			&=
			2\left\langle
			P_{x_i(t)x_j(t)}\nabla_{v_j}v_j(t)
			-
			\nabla_{v_i}v_i(t),
			P_{x_i(t)x_j(t)}v_j(t)-v_i(t)
			\right\rangle_{x_i(t)}
			\\
			&\quad
			-
			\|P_{x_i(t)x_j(t)}v_j(t)+v_i(t)\|_{x_i(t)}^2
			\frac{
				\langle x_i(t),v_j(t)\rangle_L
				+
				\langle v_i(t),x_j(t)\rangle_L
			}
			{
				1-\langle x_i(t),x_j(t)\rangle_L
			}
			\\
			&\quad
			+
			\left(
			\|v_j(t)\|_{x_j(t)}^2
			-
			\|v_i(t)\|_{x_i(t)}^2
			\right)
			\frac{
				\langle x_i(t),v_j(t)\rangle_L
				-
				\langle v_i(t),x_j(t)\rangle_L
			}
			{
				1-\langle x_i(t),x_j(t)\rangle_L
			}.
		\end{aligned}
		\]
	\end{enumerate}
\end{lemma}

By Lemma~\ref{L3.3}, the injectivity radius of $\mathbb H^d$ is infinite.
Therefore, the logarithmic vector $\log_{x_i(t)}x_j(t)$ and the parallel
transport map $P_{x_i(t)x_j(t)}$ are globally well-defined for all
$i,j\in[N]$ and all $t\ge0$. Hence, the \textit{a priori} assumption
$(\mathcal P_1)$ is automatically satisfied on $\mathbb H^d$.

It remains to verify the uniform time regularity required in
$(\mathcal P_2)$. The seventh identity in Lemma~\ref{L3.3} will be employed for this
purpose. Indeed, together with the uniform velocity estimate in
Corollary~\ref{C3.1} and the uniform bound for the weighted logarithmic
interactions obtained in Lemma \ref{L3.2}, it gives uniform control of the time variation of
the transported velocity discrepancy on the hyperboloid model.

\begin{theorem}
	\label{T3.2}
	\emph{(Removal of the \textit{a priori} assumptions on $\mathbb H^d$)}
    Let the underlying manifold be $\mathbb H^d$, realized as the hyperboloid
    model, and assume that $(\mathcal F_2)$ holds with $k=-1$. Let
    $\{(x_i(t),v_i(t))\}_{i=1}^N$ be a global solution of \eqref{Main}. Then
    the solution exhibits
	interaction-weighted asymptotic velocity alignment. More precisely, we have
	\[
	\lim_{t\to\infty}
	\max_{i,j\in[N]}
	\phi(d_{\mathbb H}(x_i(t),x_j(t)))
	\|P_{x_i(t)x_j(t)}v_j(t)-v_i(t)\|_{x_i(t)}
	=
	0.
	\]
\end{theorem}

\begin{proof}
	By the third assertion of Lemma~\ref{L3.3}, the injectivity radius of $\mathbb H^d$ is infinite.
	Hence, the logarithmic interactions and the parallel transport maps are
	globally well-defined. Thus, the assumption $(\mathcal P_1)$ holds
	automatically on $\mathbb H^d$.
	We next verify $(\mathcal P_2)$. By Corollary~\ref{C3.1}, there exists a
	constant
	\[
	V_0:=\sqrt{2\mathcal E(0)}
	\]
	such that
	\[
	\|v_i(t)\|_{x_i(t)}\le V_0,
	\quad
	\text{for all } i\in[N] \text{ and } t\ge0.
	\]
Moreover, Lemma~\ref{L3.2} and the equation of motion imply the uniform
acceleration bound. Setting
\[
A_0:=2\kappa M V_0,
\qquad
M=\sup_{r\ge0} \big(r\coth r\,\phi(r)\big),
\]
we have
\[
\|\nabla_{v_i}v_i(t)\|_{x_i(t)}\le A_0,
\quad
\text{for all } i\in[N] \text{ and } t\ge0.
\]
	Fix $i,j\in[N]$. Since the case $i=j$ is trivial, we assume $i\ne j$. From the seventh
	identity in Lemma~\ref{L3.3}, we recall
	\[
	\begin{aligned}
		&
		\frac{d}{dt}
		\|P_{x_i(t)x_j(t)}v_j(t)-v_i(t)\|_{x_i(t)}^2
		\\
		&=
		2\left\langle
		P_{x_i(t)x_j(t)}\nabla_{v_j}v_j(t)
		-
		\nabla_{v_i}v_i(t),
		P_{x_i(t)x_j(t)}v_j(t)-v_i(t)
		\right\rangle_{x_i(t)}
		\\
		&\quad
		-
		\|P_{x_i(t)x_j(t)}v_j(t)+v_i(t)\|_{x_i(t)}^2
		\frac{
			\langle x_i(t),v_j(t)\rangle_L
			+
			\langle v_i(t),x_j(t)\rangle_L
		}
		{
			1-\langle x_i(t),x_j(t)\rangle_L
		}
		\\
		&\quad
		+
		\left(
		\|v_j(t)\|_{x_j(t)}^2
		-
		\|v_i(t)\|_{x_i(t)}^2
		\right)
		\frac{
			\langle x_i(t),v_j(t)\rangle_L
			-
			\langle v_i(t),x_j(t)\rangle_L
		}
		{
			1-\langle x_i(t),x_j(t)\rangle_L
		}.
	\end{aligned}
	\]
	We estimate each term on the right-hand side. Since parallel transport is an
	isometry,
	\[
	\|P_{x_i(t)x_j(t)}v_j(t)-v_i(t)\|_{x_i(t)}
	\le
	\|v_j(t)\|_{x_j(t)}+\|v_i(t)\|_{x_i(t)}
	\le
	2V_0
	\]
	and
	\[
	\|P_{x_i(t)x_j(t)}v_j(t)+v_i(t)\|_{x_i(t)}
	\le
	2V_0.
	\]
	Therefore, we see
	\begin{align}\label{C4}
	\begin{aligned}
		&
		2\left|
		\left\langle
		P_{x_i(t)x_j(t)}\nabla_{v_j}v_j(t)
		-
		\nabla_{v_i}v_i(t),
		P_{x_i(t)x_j(t)}v_j(t)-v_i(t)
		\right\rangle_{x_i(t)}
		\right|
		\\
		&\quad \le
		2\left(
		\|\nabla_{v_j}v_j(t)\|_{x_j(t)}
		+
		\|\nabla_{v_i}v_i(t)\|_{x_i(t)}
		\right)
		\|P_{x_i(t)x_j(t)}v_j(t)-v_i(t)\|_{x_i(t)}
		\\
		&\quad \le
		8A_0V_0.
	\end{aligned}
	\end{align}
It remains to estimate the two fractional terms appearing in the identity of
Lemma~\ref{L3.3}. For the sake of simplicity, we set
\[
r(t):=d_{\mathbb H}(x_i(t),x_j(t)).
\]
Since
\[
-\langle x_i(t),x_j(t)\rangle_L=\cosh r(t),
\]
we have
\[
1-\langle x_i(t),x_j(t)\rangle_L
=
1+\cosh r(t).
\]
Differentiating
\[
-\langle x_i(t),x_j(t)\rangle_L=\cosh r(t)
\]
with respect to time gives
\[
-\langle v_i(t),x_j(t)\rangle_L
-
\langle x_i(t),v_j(t)\rangle_L
=
\sinh r(t)\frac{d}{dt}r(t).
\]
Hence, we find
\[
\begin{aligned}
	&
	\left|
	\frac{
		\langle x_i(t),v_j(t)\rangle_L
		+
		\langle v_i(t),x_j(t)\rangle_L
	}
	{
		1-\langle x_i(t),x_j(t)\rangle_L
	}
	\right|
 =
	\frac{\sinh r(t)}{1+\cosh r(t)}
	\left|
	\frac{d}{dt}r(t)
	\right|
 \le
	\left|
	\frac{d}{dt}r(t)
	\right|.
\end{aligned}
\]
Moreover, using the following relation:
\[
\left|
\frac{d}{dt}r(t)
\right|
\le
\|P_{x_i(t)x_j(t)}v_j(t)-v_i(t)\|_{x_i(t)},
\]
we obtain
\[
\left|
\frac{
	\langle x_i(t),v_j(t)\rangle_L
	+
	\langle v_i(t),x_j(t)\rangle_L
}
{
	1-\langle x_i(t),x_j(t)\rangle_L
}
\right|
\le
\|P_{x_i(t)x_j(t)}v_j(t)-v_i(t)\|_{x_i(t)}
\le
2V_0.
\]
Similarly, by considering the distance derivative with one endpoint fixed, we
obtain
\[
\frac{|\langle x_i(t),v_j(t)\rangle_L|}
{1-\langle x_i(t),x_j(t)\rangle_L}
\le
\frac{\sinh r(t)}{1+\cosh r(t)}
\|v_j(t)\|_{x_j(t)}
\le
\|v_j(t)\|_{x_j(t)}
\le
V_0,
\]
and
\[
\frac{|\langle v_i(t),x_j(t)\rangle_L|}
{1-\langle x_i(t),x_j(t)\rangle_L}
\le
\frac{\sinh r(t)}{1+\cosh r(t)}
\|v_i(t)\|_{x_i(t)}
\le
\|v_i(t)\|_{x_i(t)}
\le
V_0.
\]
Thus,
\[
\left|
\frac{
	\langle x_i(t),v_j(t)\rangle_L
	-
	\langle v_i(t),x_j(t)\rangle_L
}
{
	1-\langle x_i(t),x_j(t)\rangle_L
}
\right|
\le
2V_0.
\]
	Consequently, the second term is bounded by
	\begin{align}\label{C5}
	\|P_{x_i(t)x_j(t)}v_j(t)+v_i(t)\|_{x_i(t)}^2
	\left|
	\frac{
		\langle x_i(t),v_j(t)\rangle_L
		+
		\langle v_i(t),x_j(t)\rangle_L
	}
	{
		1-\langle x_i(t),x_j(t)\rangle_L
	}
	\right|
	\le
	8V_0^3,
	\end{align}
	and the third term is bounded by
	\begin{align}\label{C6}
	\left|
	\|v_j(t)\|_{x_j(t)}^2
	-
	\|v_i(t)\|_{x_i(t)}^2
	\right|
	\left|
	\frac{
		\langle x_i(t),v_j(t)\rangle_L
		-
		\langle v_i(t),x_j(t)\rangle_L
	}
	{
		1-\langle x_i(t),x_j(t)\rangle_L
	}
	\right|
	\le
	4V_0^3.
	\end{align}
	Combining \eqref{C4}, \eqref{C5}, and \eqref{C6}, we estimate
	\[
	\left|
	\frac{d}{dt}
	\|P_{x_i(t)x_j(t)}v_j(t)-v_i(t)\|_{x_i(t)}^2
	\right|
	\le
	8A_0V_0+12V_0^3.
	\]
	Thus $(\mathcal P_2)$ holds with
\[
C=8A_0V_0+12V_0^3,
\]
uniformly for all $i,j\in[N]$ and $t\ge0$.
	Therefore, all assumptions of Theorem~\ref{T3.1} are satisfied without imposing
	$(\mathcal P_1)$ and $(\mathcal P_2)$ as \textit{a priori} assumptions.
	Applying Theorem~\ref{T3.1} yields
	\[
	\lim_{t\to\infty}
	\max_{i,j\in[N]}
	\phi(d_{\mathbb H}(x_i(t),x_j(t)))
	\|P_{x_i(t)x_j(t)}v_j(t)-v_i(t)\|_{x_i(t)}
	=
	0.
	\]
	This completes the proof.
\end{proof}

\section{Conclusion}
\label{sec:4}
We introduced a logarithmic velocity alignment model on Riemannian manifolds in
which pairwise coupling is generated by the covariant time derivative of
logarithmic displacement. The resulting energy identity reveals a direct link
between curvature and dissipation: nonpositive sectional curvature gives the
index form the sign needed for kinetic-energy decay, while a lower curvature
bound controls the growth of the logarithmic interaction through comparison
estimates.

Within this framework, we proved interaction-weighted asymptotic velocity
alignment for global solutions under two auxiliary assumptions ensuring global
well-definedness of the logarithmic interactions and uniform time regularity of
transported velocity discrepancies. On hyperbolic space, the global geometry,
together with the energy and interaction estimates, verifies both assumptions
directly. This yields the corresponding alignment theorem without imposing them
separately.

Natural extensions include discrete-time schemes that preserve the geometric
dissipation mechanism, mean-field and kinetic limits on Riemannian manifolds,
and models with additional interaction or external-force terms. These questions
require controlling the logarithmic coupling beyond the finite-dimensional
energy estimates developed here.

\end{document}